\documentclass[12pt]{amsart}
\usepackage{amsmath,amssymb,amsthm}
\usepackage{mathtools}
\usepackage{scalerel}
\usepackage{enumerate}
\usepackage{hyperref}
\hypersetup{colorlinks=true,citecolor=blue,linkcolor=red}
\usepackage{cite}
\usepackage{color}
\usepackage[usestackEOL]{stackengine}
\usepackage[a4paper,width=16.5cm,top=3cm,bottom=2.8cm]{geometry}

\newtheorem{theorem}{Theorem}[section]

\newtheorem{lemma}[theorem]{Lemma}

\newtheorem{remark}[theorem]{Remark}

\newtheorem{question}[theorem]{Question}

\theoremstyle{plain}
\newtheorem*{theorem*}{Theorem}
\newtheorem*{question*}{Question}
\newtheorem*{example*}{Example}

\def\dashint{\,\ThisStyle{\ensurestackMath{%
			\stackinset{c}{.2\LMpt}{c}{.5\LMpt}{\SavedStyle-}{\SavedStyle\phantom{\int}}}%
		\setbox0=\hbox{$\SavedStyle\int\,$}\kern-\wd0}\int}
\newcommand{\rn}{\mathbb{R}^n}

\newcommand{\dyprodhp}{H_d^p(\mathbb{R}\otimes\mathbb{R})}

\newcommand{\dyprodhpn}[1]{\|#1\|_{\dyprodhp}}

\newcommand{\dyprodhq}{H_d^q(\mathbb{R}\otimes\mathbb{R})}

\newcommand{\dyprodhr}{H_d^r(\mathbb{R}\otimes\mathbb{R})}
\newcommand{\dyprodhrn}[1]{\|#1\|_{\dyprodhr}}

\newcommand{\dyprodbmo}{\text{BMO}_d(\mathbb{R}\otimes\mathbb{R})}

\newcommand{\dyprodbmon}[1]{\|#1\|_{\dyprodbmo}}

\newcommand{\avr}[2]{\left\langle#1\right\rangle_{#2}}

\newcommand{\lprtwo}{L^p(\mathbb{R}^2)}
\newcommand{\lprtwonorm}[1]{\|#1\|_{\lprtwo}}

\newcommand{\lrrtwo}{L^r(\mathbb{R}^2)}
\newcommand{\lrrtwonorm}[1]{\|#1\|_{\lrrtwo}}

\begin{document}
	
	\title{Notes on Bi-parameter Paraproducts}
	\author{Shahaboddin Shaabani}
	\address{Department of Mathematics\\University of Toronto}
	\email{shahaboddin.shaabani@utoronto.ca}	
	\keywords{Hardy spaces, Paraproducts}
	\subjclass{42B30, 42B35}
	\date{}
	\begin{abstract}
In this note, we investigate the sharpness of existing bounds for various types of bi-parameter paraproducts acting between product Hardy spaces in the dyadic setting. We show that these bounds are sharp in most cases but fail to be so in one particular instance.
	\end{abstract}
	\maketitle

	\section{Introduction}
	
This work is a sequel to our previous two papers on the boundedness properties of paraproducts \cite{MR4952996,MR4939674}. Here, we are concerned with sharpness of the existing bounds for the remaining dyadic bi-parameter paraproducts not studied in \cite{MR4952996}. For various applications and properties of paraproducts, we refer the reader to \cite{MR2682821,MR3375865,MR1853518,MR2730492,MR2134868,MR2320408,MR3052499,MR2313844}.\\ 

Let us begin by recalling the one-parameter dyadic paraproducts, three bilinear forms defined by
\begin{equation*}
	\pi^{\varepsilon}(f,g)(x):=\sum_{I\in \mathcal{D}} \avr{f, h_I^{\varepsilon_1}}{}\avr{g, h_I^{\varepsilon_2}}{} h_I^{\varepsilon_1+\varepsilon_2}(x), \quad \varepsilon=(\varepsilon_1,\varepsilon_2)\in \{0,1\}^2\backslash\{(0,0)\}, \quad x\in\mathbb{R}.
\end{equation*}
In the above, $I$ stands for a dyadic interval on the line and $\mathcal{D}$ for the collection of all such intervals. Also,  $h_I$ is the $L^2$-normalized Haar wavelet associated to $I$ and $h_I^\epsilon$ is given by
\[
h_I^{\epsilon}(x):= \begin{cases}
	h_I(x) & \epsilon=1\\
	\bar{\chi}_I(x) & \epsilon=0,
\end{cases} \quad x\in\mathbb{R}, 
\]
where $\bar{\chi}_I:=\frac{\chi_I}{|I|}$. Finally, the sum $\varepsilon_1+\varepsilon_2$ is understood in mode 2. It is then easy to see that for sufficiently nice functions we have
\begin{equation}\label{paraproductexpansioninR}
	f(x)g(x)=\sum_{\varepsilon\in\{0,1\}^2\backslash\{(0,0)\}} \pi^{\varepsilon}(f,g)(x), \quad x\in\mathbb{R}.
\end{equation}
To avoid unnecessary complications, throughout this note we make the qualitative assumption that all our functions, on either $\mathbb{R}$ or $\mathbb{R}^2$, are real-valued and simple, in the sense that they are finite linear combinations of characteristic functions of dyadic intervals or squares. Since none of the bounds depend on this a priori assumption, standard limiting arguments extend our results to the general case. See also \cite{MR4952996,MR4939674}, where a different approach has been taken.\\

Fairly well-known arguments show that, for $0<p<\infty$, the bilinear form $\pi^{(0,1)}$ is bounded from $H^p_d(\mathbb{R})\times BMO_d(\mathbb{R})$ and $L^\infty(\mathbb{R})\times H^p_d(\mathbb{R})$ to $H^p_d(\mathbb{R})$ \cite{MR2730492,MR3375865, MR2134868,MR2320408,MR1853518, MR3052499,MR2313844}. Here, for $0<p<\infty$, $H^p_d(\mathbb{R})$ denotes the dyadic Hardy space on the line, (quasi-)normed by
\[
\|f\|_{H^p_d(\mathbb{R})}:= \|M(f)\|_{L^p(\mathbb{R})}, \quad M(f)(x):=\sup_{x\in I}|\avr{f}{I}|, \quad x\in\mathbb{R},
\]
where we use $\avr{f}{E}$ for the average of $f$ over a measurable set $E$. Also, $BMO_d(\mathbb{R})$ denotes the space of functions with uniformly bounded mean oscillation on dyadic intervals, i.e.
\[
\|f\|_{BMO_d(\mathbb{R})}:= \sup_{I\in\mathcal{D}} \big(\frac{1}{|I|} \sum_{I'\subseteq I} \avr{f,h_{I'}}{}^2\big)^{\frac{1}{2}}<\infty.
\]
We refer the reader to \cite{MR1320508} for more on these spaces.\\

Because of the symmetry $\pi^{(0,1)}(f,g)=\pi^{(1,0)}(g,f)$, a similar result holds for $\pi^{(1,0)}$. In addition, both forms are bounded from $H^p_d(\mathbb{R})\times H^r_d(\mathbb{R})$ to $H^q_d(\mathbb{R})$, with $\frac{1}{q}=\frac{1}{p}+\frac{1}{r}$ and $0<r,p,q<\infty$. These two results also hold for $\pi^{(1,1)}$, but only in the reflexive range of exponents, i.e., where $1<r,p,q<\infty$. Regarding the sharpness of these boundedness properties, a natural approach is to freeze one of the inputs and study the operator norm of the resulting linear operators. Because of the symmetries, it is enough to study the linear operators
\begin{equation}\label{Defoneparameterparaproducts}
	\pi_g(f)= \sum_{I\in \mathcal{D}} \avr{f}{I} \avr{g,h_I}{} h_I, \quad
	\pi'_g(f)= \sum_{I\in \mathcal{D}} \avr{f,h_I}{} \avr{g,h_I}{} \bar{\chi}_I, \quad
	\pi''_g(f)=\sum_{I\in \mathcal{D}} \avr{f,h_I}{}\avr{g}{I} h_I.
\end{equation}
We may also exclude $\pi'_g$ from our study, since it is the adjoint of $\pi_g$ and therefore has the same norm in the Banach range of exponents. In \cite{MR2381883}, it was shown that 
\[
\|\pi_g\|_{L^p(\mathbb{R})\to L^p(\mathbb{R})}\simeq \|g\|_{BMO_d(\mathbb{R})}, \quad 1<p<\infty,
\]
and in \cite{MR4947121}, the authors established the equivalence
\[
\|\pi_g\|_{L^p(\mathbb{R})\to L^q(\mathbb{R})}\simeq \|g\|_{L^r(\mathbb{R})}, \quad \frac{1}{q}=\frac{1}{p}+\frac{1}{r}, \quad 1<p,r,q<\infty.
\]
Recently, in \cite{MR4939674}, we showed that both results extend to the full range of exponents, provided that Lebesgue spaces are replaced with dyadic Hardy spaces. The operator $\pi''_g$ is simpler to deal with, and one can easily verify that 
\[
\|\pi''_g\|_{H_d^p(\mathbb{R})\to H_d^p(\mathbb{R})}\simeq \|g\|_{L^\infty(\mathbb{R})}, \quad 0<p<\infty.
\]
It is enough to recall the square function characterization of dyadic Hardy spaces, stating that
\[
\|f\|_{H_d^p(\mathbb{R})}\simeq \|S(f)\|_{L^p(\mathbb{R})}, \quad 0<p<\infty, \quad S(f):=\big(\sum_{I\in \mathcal{D}}\avr{f, h_I}{}^2\bar{\chi}_I\big)^{\frac{1}{2}}.
\]	
As expected, the $H_d^p(\mathbb{R})$-to-$H_d^q(\mathbb{R})$ norm of $\pi''_g$ is also comparable to $\|g\|_{H_d^r(\mathbb{R})}$ in the full range of exponents, when $\frac{1}{q}=\frac{1}{p}+\frac{1}{r}$. In the reflexive range, this result was established in \cite{MR1438787} (Theorem 12.2, p.~128), and in Theorem \ref{pi222theorem} we provide a proof of this in the product setting, to which we now turn.\\

Let us recall the family of bi-parameter paraproducts: the nine bilinear forms arising from the product of two functions on the plane, expanded in the rectangular Haar basis. Let
\[
h_I\otimes h_J(x,y)=h_I(x)h_J(y), \quad (x,y)\in\mathbb{R}^2,
\]
be the Haar wavelet associated with the dyadic rectangle $I\times J$, and let $f$ and $g$ be two functions on $\mathbb{R}^2$. For $\varepsilon=(\varepsilon_1, \varepsilon_2)$ with $\varepsilon_1=(\varepsilon_{11},\varepsilon_{12})$ and $\varepsilon_2=(\varepsilon_{21},\varepsilon_{22})\in\{0,1\}^2\backslash\{(0,0)\}$, the bi-parameter paraproduct $\pi^{\varepsilon}$ is defined as the tensor product of $\pi^{\varepsilon_1}$ and $\pi^{\varepsilon_2}$, i.e.,
\begin{equation}\label{defitionofparaproducts}
	\pi^{\varepsilon}(f,g):=\pi^{\varepsilon_1}\otimes\pi^{\varepsilon_2}(f,g)
	= \sum_{I,J\in\mathcal{D}} \avr{f,h_I^{\varepsilon_{11}}\otimes h_J^{\varepsilon_{21}}}{} \avr{g,h_I^{\varepsilon_{12}}\otimes h_J^{\varepsilon_{22}}}{} h_I^{\varepsilon_{11}+\varepsilon_{12}}\otimes h_J^{\varepsilon_{21}+\varepsilon_{22}}.
\end{equation}	
Similarly to \eqref{paraproductexpansioninR}, we have
\[
f(x,y)g(x,y)=\sum_{\substack{\varepsilon=(\varepsilon_1,\varepsilon_2)\\\varepsilon_1,\varepsilon_2\neq (0,0)}} \pi^{\varepsilon}(f,g)(x,y), \quad (x,y)\in \mathbb{R}^2,
\]
which directly follows from \eqref{paraproductexpansioninR} when $f=f_1\otimes f_2$ and $g=g_1\otimes g_2$, and from linearity in the general case. When $\varepsilon_1=\varepsilon_2$, the bilinear form $\pi^{\varepsilon}$ is called an ``unmixed'' paraproduct, since in such cases the cancellative terms $h_I, h_J$ and the non-cancellative terms $\bar{\chi}_I, \bar{\chi}_J$ are separated from each other in all three terms of the sum in \eqref{defitionofparaproducts}. The other forms are referred to as ``mixed'' paraproducts. \\

To discuss boundedness of these forms and the sharpness of the existing results, we take a similar approach as explained before. Considering the symmetries in $f$ and $g$, and in $x$ and $y$, and excluding the adjoint operators, one identifies four different linear operators, listed below.
	\begin{align*}
		&\pi^1_g(f):=\pi^{(0,1)}\otimes \pi^{(0,1)}(f,g)=\sum_{I,J\in\mathcal{D}}\avr{f, \bar{\chi}_I\otimes \bar{\chi}_J}{} \avr{g, h_I\otimes h_J}{}h_I\otimes h_J,\\
		& \pi^2_g(f):=\pi^{(1,0)}\otimes \pi^{(1,0)}(f,g)=\sum_{I,J\in\mathcal{D}}\avr{f,h_I\otimes h_J}{} \avr{g,\bar{\chi}_I\otimes \bar{\chi}_J }{}h_I\otimes h_J\\
		&\pi^3_g(f):=\pi^{(0,1)}\otimes \pi^{(1,0)}(f,g)=\sum_{I,J\in\mathcal{D}}\avr{f, \bar{\chi}_I\otimes h_J}{}\avr{g, h_I\otimes\bar{\chi}_J}{}h_I\otimes h_J\\	
		& \pi^4_g(f)=\pi^{(0,1)}\otimes \pi^{(1,1)}(f,g)=\sum_{I,J\in\mathcal{D}}\avr{f,\bar{\chi}_I\otimes h_J}{}\avr{g, h_I\otimes h_J}{} h_I\otimes \bar{\chi}_J.	
	\end{align*}
In the rest of this note, we discuss the boundedness of the above linear operators acting between different product Hardy spaces. We refer the reader to Theorem~\ref{pi222theorem}, Theorem~\ref{pi3theorem}, and the example presented at the end of Section~3 for a quick overview of the new results of this note.\\

Before proceeding, it is convenient to simplify the notation. In $\mathbb{R}^2$, we use $I$ to denote a dyadic interval on the $x$-axis and $J$ for such intervals on the $y$-axis. From now on, we write $R = I \times J$ for a dyadic rectangle in the plane, set $h_R = h_I \otimes h_J$, and for a function $f$, define $f_R = \langle f, h_R \rangle$. In addition, for a function $f$ on $\mathbb{R}^2$, we let
\[
f_I(y) := \langle f(\cdot, y), h_I \rangle, \qquad f_J(x) := \langle f(x, \cdot), h_J \rangle.
\]

\section{Unmixed Paraproducts}
 Let us begin with the operator $\pi^1_g$, which was the main object of our recent work \cite{MR4952996}, and is given by
 \[
 \pi^1_g(f):=\sum_{R} \avr{f}{R} g_R h_R.
 \]
 There, we have shown that the behavior of this operator and its one-parameter analogue $\pi_g$, defined in \eqref{Defoneparameterparaproducts}, are identical. More precisely, we have shown that
 \begin{align*}
 	&\|\pi^1_g\|_{\dyprodhp\to \dyprodhp} \simeq \dyprodbmon{g}, \quad 0<p<\infty,\\
 	&\|\pi^1_g\|_{\dyprodhp\to \dyprodhq} \simeq \dyprodhrn{g}, \quad \frac{1}{q}=\frac{1}{p}+\frac{1}{r}, \quad 0<p,r,q<\infty,
 \end{align*}
 where in the above $\dyprodhp$ stands for the bi-parameter dyadic Hardy space, (quasi-)normed by
 \[
 \dyprodhpn{f}:=\lprtwonorm{M(f)}, \quad M(f)(x,y):=\sup_{(x,y)\in R}|\avr{f}{R}|, \quad (x,y)\in \mathbb{R}^2,
 \]
 and $\dyprodbmo$ denotes the dyadic product BMO, i.e., the space of functions with
 \[
 \dyprodbmon{f}:=\sup_{\Omega}\big(\frac{1}{|\Omega|}\sum_{R\subseteq\Omega} f_R^2 \big)^\frac{1}{2}<\infty,
 \]
where the supremum is taken over all open sets of finite positive measure. Since it will be clear from the context, we use the same notation $M$ for the dyadic maximal operator on the line and the strong operator defined above. We also refer the reader to \cite{MR4952996} for the reason behind the above equivalences, and to \cite{MR0539351,MR0602392,MR1320508,MR0864369,MR766959} for more on Hardy spaces and BMO in the product setting. See also \cite{MR2174914,MR1892177}. The adjoint of this operator, which is unmixed too, is given by
\[
(\pi^{1}_g)^t(f)=\pi^{(1,1)}\otimes \pi^{(1,1)}(f,g)=\sum_{R} f_R g_R \frac{\chi_R}{|R|},
\]
and therefore it satisfies similar properties, but only in the Banach range of spaces.\\

Since $\pi^1_g$ is well-understood, we turn to the second operator, $\pi^2_g$, which is simply defined by
\[
\pi^2_g(f):=\sum_{R} f_R \avr{g}{R} h_R,
\]
and satisfies the known bounds
\begin{align}
	&\|\pi^2_g\|_{\dyprodhp\to \dyprodhp} \lesssim \|g\|_{L^\infty(\mathbb{R}^2)}, \quad 0<p<\infty, \label{pig2hptohp}\\
	&\|\pi^2_g\|_{\dyprodhp\to \dyprodhq} \lesssim \dyprodhrn{g}, \quad \frac{1}{q}=\frac{1}{p}+\frac{1}{r}, \quad 0<p,r,q<\infty, \label{pig2hptohq}
\end{align}
which can be easily verified after recalling the square function characterization of $\dyprodhp$, i.e., the fact that
\[
\dyprodhpn{f} \simeq \lprtwonorm{S(f)}, \quad 0<p<\infty, \quad S(f):=\Big(\sum_{R} |f_R|^2 \frac{\chi_R}{|R|}\Big)^\frac{1}{2},
\]
\cite{MR0602392,MR1320508}.
Now, the above bounds are simple consequences of the crucial pointwise inequality
\[
S(\pi^2_g(f))\le S(f) M(g),
\]
the square function characterization of product Hardy spaces, and H\"older's inequality \cite{MR2134868,MR2320408, MR3052499,MR2313844}.\\ 

Regarding the sharpness of \eqref{pig2hptohp} and \eqref{pig2hptohq}, we could not find anything in the literature. Thus, as our first task, we show that, just like in the one-parameter setting, both of these inequalities are indeed equivalences. Before doing so, let us briefly recall the notion of Carleson families, the John-Nirenberg lemma in the product setting, and a weak form of the Fefferman-C\'ordoba covering lemma for rectangles.\\

A collection of rectangles (with sides parallel to the axes), $\mathcal{C}$, is called $\Lambda$-Carleson if, for every open set $\Omega$, there holds
\[
\sum_{\substack{R\subseteq \Omega\\ R\in \mathcal{C}}} |R| \le \Lambda |\Omega|.
\]
As is well-known, this condition is equivalent to $\eta$-sparseness with $\eta = \Lambda^{-1}$ \cite{MR3893778,honig2025optimizationalgorithmscarlesonsparse}. The sparseness condition means that every element $R$ of the family has a piece $E_R \subset R$ of density at least $\eta$ such that all these pieces are pairwise disjoint. Then the John-Nirenberg lemma states that an $\eta$-sparse family of rectangles is essentially a disjoint family in the sense that
\[
\lprtwonorm{\sum_{\substack{R\subseteq \Omega\\ R\in \mathcal{C}}}\chi_R} \lesssim_{p,\eta} |\Omega|^{\frac{1}{p}}, \quad 0<p<\infty.
\]
Here, the case $0<p\le 1$, follows simply from H\"older's inequality and we refer the reader to \cite{MR3052499} for the proof of the general case. Finally, it follows from the Fefferman-C\'ordoba covering lemma that, for any collection of rectangles $\mathcal{C}$, it is possible to extract a $\frac{1}{2}$-sparse sub-collection $\mathcal{C}'$ such that
\[
|\cup_{R\in \mathcal{C}} R| \simeq |\cup_{R'\in \mathcal{C}'} R'|.
\]
See \cite{MR0379785,MR0864369, MR3052499} for a detailed exposition of these.
\begin{theorem}\label{pi222theorem}
	For any function $g$, both bounds in \eqref{pig2hptohp} and \eqref{pig2hptohq} are indeed equivalences.
\end{theorem}	

\begin{proof}
	The first case follows by plugging $f=h_R$ into the operator and using Lebesgue's differentiation theorem. We now show that \eqref{pig2hptohq} is sharp. To this end, normalize $g$ such that 
	\[
	\|\pi^2_g\|_{\dyprodhp\to \dyprodhq}=1.
	\]
	Then we need to show that $\dyprodhrn{g} \lesssim 1$, and to do so, we build a proper test function by looking at the level sets of $M(g)$.\\ 
	
	For an integer $k$, observe that the level set $\{M(g)>2^k\}$ is a union of rectangles $R$ with the property that $|\avr{g}{R}|>2^k$. Apply the above-mentioned covering lemma to these rectangles and extract a $\frac{1}{2}$-sparse sub-collection $\mathcal{C}_k$, whose total measure is proportional to $|\{M(g)>2^k\}|$. Then consider $\mathcal{C}=\cup_{k\in\mathbb{Z}} \mathcal{C}_k$, and for each rectangle $R \in \mathcal{C}$, let $\lambda(R)$ be the largest $k$ for which $R \in \mathcal{C}_k$. Based on our a priori assumption on $g$, this function has compact support and is bounded. Therefore, each $\mathcal{C}$ is finite and $\lambda(R)$ is well-defined. Next, consider the test function
	\[
	f = \sum_{R \in \mathcal{C}} 2^{t \lambda(R)} |R|^{\frac{1}{2}} h_R, \quad t = \frac{r}{p},
	\] 
	satisfying
	\begin{equation}\label{estimateforhpnormfff}
		\dyprodhpn{f} \lesssim \dyprodhrn{g}^{\frac{r}{p}}.
	\end{equation}
To see this, we estimate the $L^p(\mathbb{R}^2)$-norm of $S(f)$, and in doing so we consider two separate cases. First, suppose $0<p\le 2$, in which case we take advantage of sub-linearity and obtain
\[
S(f)^p \le \sum_{k\in\mathbb{Z}} 2^{ptk} \big(\sum_{R \in \mathcal{C}_k} \chi_R \big)^{\frac{p}{2}},
\]
which, after integrating and applying the John-Nirenberg lemma, gives
\[
\lprtwonorm{S(f)}^p \le \sum_{k\in\mathbb{Z}} 2^{ptk} \int \big(\sum_{R \in \mathcal{C}_k} \chi_R \big)^{\frac{p}{2}} \lesssim \sum_{k\in\mathbb{Z}} 2^{ptk} \big|\cup_{R \in \mathcal{C}_k} R\big|.
\]
Now, recall that $|\cup_{R \in \mathcal{C}_k} R| \simeq |\{M(g) > 2^k\}|$, and apply the layer-cake formula to obtain
\[
\lprtwonorm{S(f)}^p \lesssim \sum_{k\in\mathbb{Z}} 2^{rk} |\{M(g) > 2^k\}| \simeq \lrrtwonorm{M(g)}^r,
\]
which is the claimed inequality. Next, consider the case $2<p<\infty$, in which we appeal to duality. Let $\varphi$ be a function with $\|\varphi\|_{L^{(\frac{p}{2})'}(\mathbb{R}^2)} = 1$, such that
\[
\|S(f)\|_{L^p(\mathbb{R}^2)}^2 = \|S(f)^2\|_{L^{\frac{p}{2}}(\mathbb{R}^2)} = \int S(f)^2 \varphi,
\]
where, as usual, $(\frac{p}{2})'$ is the H\"older conjugate of $\frac{p}{2}$. Then we have
\[
\|S(f)\|_{L^p(\mathbb{R}^2)}^2 = \int \sum_{R \in \mathcal{C}} 2^{2t\lambda(R)} \chi_R \varphi 
= \int \sum_{R \in \mathcal{C}} 2^{2t\lambda(R)} \avr{\varphi}{R} |R| 
\le \sum_{k\in\mathbb{Z}} 2^{2tk} \sum_{R \in \mathcal{C}_k} |\avr{\varphi}{R}| |R|.
\]
At this point, recall that each collection $\mathcal{C}_k$ is $\frac{1}{2}$-sparse, and thus for each rectangle $R \in \mathcal{C}_k$, we may find a subset $E_R \subset R$ with $|R| \le 2 |E_R|$, such that the sets $E_R$ are disjoint. Therefore, we may continue to estimate the last term by
\[
\|S(f)\|_{L^p(\mathbb{R}^2)}^2 \lesssim \sum_{k\in\mathbb{Z}} 2^{2tk} \sum_{R \in \mathcal{C}_k} |\avr{\varphi}{R}| |E_R| 
\le \sum_{k\in\mathbb{Z}} 2^{2tk} \int_{\cup_{R \in \mathcal{C}_k} E_R} M(\varphi) 
\le \sum_{k\in\mathbb{Z}} 2^{2tk} \int_{\{M(g) > 2^k\}} M(\varphi),
\]
which implies that
\[
\|S(f)\|_{L^p(\mathbb{R}^2)}^2 \lesssim \int M(\varphi) M(g)^{2t} 
\le \|M(\varphi)\|_{L^{(\frac{p}{2})'}(\mathbb{R}^2)} \|M(g)\|_{L^r(\mathbb{R}^2)}^{\frac{2r}{p}} 
\lesssim \|M(g)\|_{L^r(\mathbb{R}^2)}^{\frac{2r}{p}},
\]
again yielding the claimed inequality. In the above, we used the boundedness of the strong maximal operator $M$ on $L^{\frac{p}{2}}$, which is allowed since $\frac{p}{2} > 1$.\\

Next, we apply the operator $\pi^2_g$ to $f$ and obtain
\[
F = \pi^2_g(f) = \sum_{R \in \mathcal{C}} 2^{t \lambda(R)} \avr{g}{R} |R|^{\frac{1}{2}} h_R,
\]
which has two properties. 
First,
\[
S(F) \ge 2^{t \lambda(R)} |\avr{g}{R}| \chi_R, \quad R \in \mathcal{C},
\]
implying that
\begin{equation}\label{firstpropertyofFFF}
	S(F)(x,y) > 2^{(1+t)k}, \quad (x,y) \in \cup_{R \in \mathcal{C}_k} R, \quad k \in \mathbb{Z}.
\end{equation}
Second, since $F = \pi^2_g(f)$ and we have established that $\dyprodhpn{f} \lesssim \dyprodhrn{g}^{\frac{r}{p}}$, we must have
\begin{equation}\label{secondpropertyofFFF}
	\|S(F)\|_{L^q(\mathbb{R}^2)} \lesssim \dyprodhrn{g}^{\frac{r}{p}}.
\end{equation}
Therefore, from \eqref{firstpropertyofFFF} and the layer-cake formula, we obtain
\[
\int M(g)^r \simeq \sum_{k \in \mathbb{Z}} 2^{rk} |\{M(g) > 2^k\}| \simeq \sum_{k \in \mathbb{Z}} 2^{rk} |\cup_{R \in \mathcal{C}_k} R| \le \sum_{k \in \mathbb{Z}} 2^{rk} |\{S(F) > 2^{(1+t)k}\}|,
\]
which, after recalling that $1+t = \frac{r}{q}$, implies
\[
\int M(g)^r \lesssim \int S(F)^q.
\]
Finally, combining this with \eqref{secondpropertyofFFF} yields
\[
\lrrtwonorm{M(g)}^{\frac{r}{q}} \lesssim 	\|S(F)\|_{L^q(\mathbb{R}^2)}\lesssim \dyprodhrn{g}^{\frac{r}{p}},
\]
which is exactly $\dyprodhrn{g} \lesssim 1$, the desired result. The proof is now complete.
\end{proof}
\section{Mixed Paraproducts}

Now we turn our attention to the last two operators, $\pi^3_g$ and $\pi^4_g$, which are of mixed type. The analysis of such operators naturally leads one to consider the mixed square-maximal or maximal-square operators of the form
\[
S_2 M_1(f)(x,y) := \Big(\sum_{J \in \mathcal{D}} M(f_J)^2(x) \bar{\chi}_J(y)\Big)^{\frac{1}{2}}, 
\quad 
M_1 S_2(f)(x,y) := \sup_{x \in I} \Big(\sum_{J \in \mathcal{D}} |\avr{f_J}{I}|^2 \bar{\chi}_J(y)\Big)^{\frac{1}{2}},
\]
and, quite similarly,
\[
S_1 M_2(g)(x,y) := \Big(\sum_{I \in \mathcal{D}} M(g_I)^2(y) \bar{\chi}_I(x)\Big)^{\frac{1}{2}}, 
\quad 
M_2 S_1(g)(x,y) := \sup_{y \in J} \Big(\sum_{I \in \mathcal{D}} |\avr{g_I}{J}|^2 \bar{\chi}_I(x)\Big)^{\frac{1}{2}}.
\]
It is then not hard to see, and it is well-known, that each of these mixed operators gives another characterization of $\dyprodhp$ \cite{MR2134868,MR2320408,MR2313844,MR3052499}. This means that for $0<p<\infty$, the $L^p(\mathbb{R}^2)$-norm of each of the four functions above is comparable to the $\dyprodhp$-norm of $f$ and $g$, respectively. Let us briefly sketch the proof of these equivalences, say for $f$. First, note that the pointwise inequality 
\[
M_1 S_2(f) \le S_2 M_1(f),
\]
implies
\[
\|M_1 S_2(f)\|_{L^p(\mathbb{R}^2)} \le \|S_2 M_1(f)\|_{L^p(\mathbb{R}^2)}.
\]	
To see why
\[
\|S_2 M_1(f)\|_{L^p(\mathbb{R}^2)} \lesssim \|S(f)\|_{L^p(\mathbb{R}^2)},
\]
fix $y$ and observe that the vector-valued Fefferman-Stein inequality for the maximal operator on the line, $M$, together with the square function characterization of $L^p(\mathbb{R})$, implies that the operator $S_2 M_1$ is bounded on all $L^p(\mathbb{R})$ with $1<p<\infty$. This, combined with Fubini's theorem, establishes the above inequality for $1<p<\infty$. The case $0<p\le 1$ then follows from the atomic decomposition of $\dyprodhp$. To complete the chain of equivalences, one has to show that
\[
\|S(f)\|_{L^p(\mathbb{R}^2)} \lesssim \|M_1 S_2(f)\|_{L^p(\mathbb{R}^2)},
\]	
which, after using Fubini's theorem and fixing $y$ again, follows from the equivalence of the square function and maximal characterizations of $H_d^p(\mathbb{R}, l^2)$.	
\subsection{The Operator $\pi^3_g$}	
Now, we continue with the boundedness properties of the first mixed paraproduct, $\pi^3_g$, given by
\[
\pi^3_g(f) := \sum_{I,J \in \mathcal{D}} \avr{f_J}{I} \avr{g_I}{J} h_I \otimes h_J.
\]
For this operator, just like for $\pi^1_g$ and $\pi^2_g$, there holds that
\begin{equation}\label{Holdersforpi333}
	\|\pi^3_g\|_{\dyprodhp \to \dyprodhq} \lesssim \dyprodhrn{g}, \quad \frac{1}{q} = \frac{1}{p} + \frac{1}{r}, \quad 0<p,r,q<\infty,
\end{equation}
which is a simple consequence of the pointwise inequality
\[
S(\pi^3_g(f)) \le S_2 M_1(f) M_2 S_1(g),
\]
and the above-mentioned mixed characterizations of product Hardy spaces \cite{MR2134868,MR2320408,MR2313844,MR3052499}. Regarding this it is natural to ask:
\begin{question}
	Is it true that \eqref{Holdersforpi333} is an equivalence in the full range of exponents?
\end{question}

Theorem \ref{pi3theorem} answers this question in the positive in the reflexive range and provides a lower bound in the general case. However, the more intricate structure of the level sets of the mixed operators involved, and their interrelations, leaves the case $0<r\leq 1$ open. Nevertheless, the same theorem shows that for all $0<p<\infty$, the operator norm of $\pi^3_g$ on $\dyprodhp$ is comparable to a mixed-type norm of $g$, i.e.
\[
\Big\| \| g(x,y) \|_{BMO_d(\mathbb{R}, dx)} \Big\|_{L^{\infty}(\mathbb{R}, dy)},
\]
which is simply the essential supremum of the $BMO_d(\mathbb{R})$-norm of the horizontal slices of $g$ (a similar notation is used for a mixture of $H^r$ and $L^r$ norms in \eqref{thesecondcasssss}). See \cite{MR2134868,MR2313844}, with the above quantity replaced by the $L^{\infty}(\mathbb{R}^2)$-norm. We will establish this using the atomic decomposition theorem from \cite{MR4952996}. To this end, let us recall some simple but useful notions.\\

In the plane, a family of measurable sets $\Omega_i$, with $i=0,1,2,\ldots$, is called contracting if
\[
\Omega_{i+1} \subset \Omega_i, \quad |\Omega_{i+1}| \le \frac{1}{2} |\Omega_i|, \quad i=0,1,2,\ldots.
\]
It is then easy to see that the maximal operator associated with such a family,
\[
m(g)(z) := \sup_{\substack{z \in \Omega_i \\ i \ge 0}} \avr{|g|}{\Omega_i},
\]
is bounded on $L^p(\mathbb{R}^2)$ for $1 < p \le \infty$.\\ 

Next, recall that in the product setting a function $f$ is called an $L^s$-atom ($\max(1,p) < s < \infty$) supported on $\Omega$ if
\[
f = \sum_{R \subseteq \Omega} f_R h_R, \quad \|f\|_{L^s(\mathbb{R}^2)} \le |\Omega|^{\frac{1}{s}}.
\]
It is then simple to see that $\dyprodhpn{f} \lesssim |\Omega|^{\frac{1}{p}}$. Finally, for any function $f$ and any $0 < p < \infty$, one may find a contracting family of open sets $\{\Omega_i\}$ and $L^s$-atoms $f_i$ supported on $\Omega_i$ such that
\begin{equation}\label{atomicdecompostion}
	f = \sum_i a_i f_i, \quad \dyprodhpn{f} \simeq \Big(\sum_i a_i^p |\Omega_i| \Big)^{\frac{1}{p}}.
\end{equation}
Also, when $f$ is finite linear combination of rectangular Haar functions this sum is finite. See \cite{MR4952996} for the proof. In order to use this atomic decomposition we need to prove a simple lemma which is very useful when working with ``local operators''.

\begin{lemma}\label{locallemma}
	Let $T$ be a linear operator that is \emph{local} in the sense that it maps $L^s$-atoms supported on $\Omega$ into $L^q$-atoms supported on the same set $\Omega$ ($1<s,q<\infty$). Then $T$ is bounded on $\dyprodhp$ for $0<p<q$. Moreover, the same conclusion holds for $1<p<q$, provided that for any $L^s$-atom $f$ supported on $\Omega$, $T(f)$ is supported on $\Omega$ (not necessarily an atom) and satisfies
	\[
	\|T(f)\|_{L^q(\mathbb{R}^2)} \le |\Omega|^{\frac{1}{q}}.
	\]
\end{lemma}
\begin{proof}
	Take $f\in \dyprodhp$, and without loss of generality assume that $f$, has a finite Haar support. Then apply the atomic decomposition \eqref{atomicdecompostion} and obtain $f_i$'s and $\Omega_i$'s. We have that
	\[
	T(f)=\sum_{i}a_iT(f_i),
	\]
	implying that for $0<p\le1$, there holds
	\[
	\dyprodhpn{T(f)}^p\le \sum_{i}a_i^p\dyprodhpn{T(f_i)}^p\lesssim \sum_{i}a_i^p|\Omega_{i}|\lesssim \dyprodhpn{f}^p,
	\]
which proves the claim. To treat the case $1<p<\infty$, we appeal to duality and pair $T(f)$ with a function $\varphi$ with $\|\varphi\|_{L^{p'}(\mathbb{R}^2)}=1$. Then
\begin{align*}
	&\avr{T(f), \varphi}{}=\sum_{i}a_i\avr{T(f_i), \varphi}{}\le \sum_{i}|a_i| \avr{|T(f_i)|^q}{\Omega_{i}}^\frac{1}{q}\avr{|\varphi|^{q'}}{\Omega_{i}}^\frac{1}{q'}|\Omega_{i}|\le\sum_{i}|a_i|\avr{|\varphi|^{q'}}{\Omega_{i}}^\frac{1}{q'}|\Omega_{i}|\le 2\\
	& \sum_{i}|a_i|\avr{|\varphi|^{q'}}{\Omega_{i}}^\frac{1}{q'}|\Omega_{i}\backslash\Omega_{i+1}|\le \int m(|\varphi|^{q'})^{\frac{1}{q'}}\sum_{i}|a_i|\chi_{\Omega_{i}\backslash\Omega_{i+1}} \lesssim \big(\sum_{i}a_i^p|\Omega_{i}|\big)^{\frac{1}{p}}\lesssim \dyprodhpn{f},
\end{align*}
proving the claim. Note that in the last line we used H\"older's inequality, boundedness of the operator $m$ on $L^{\frac{p'}{q'}}$, with $q'<p'$, as well as disjointness of the sets $\Omega_{i}\backslash\Omega_{i+1}$. The proof is now complete.
\end{proof}

\begin{theorem}\label{pi3theorem}
	For any function $g$, there holds
	\begin{equation}\label{the firstcassss}
		\|\pi^3_g\|_{\dyprodhp\to \dyprodhp} \simeq \left\|\|g(x,y)\|_{BMO_d(\mathbb{R},dx)}\right\|_{L^{\infty}(\mathbb{R}, dy)}, \quad 0<p<\infty.
	\end{equation}
	Also, for $0<p,r,q<\infty$ with $\frac{1}{q}=\frac{1}{p}+\frac{1}{r}$, we have	
	\begin{equation}\label{thesecondcasssss}
		\left\|\|g(x,y)\|_{H_d^r(\mathbb{R}, dx)}\right\|_{L^r(\mathbb{R}, dy)}\lesssim \|\pi^3_g\|_{\dyprodhp\to \dyprodhq} \lesssim \dyprodhrn{g},
	\end{equation}
and therefore when $1<r<\infty$, we have $$\|\pi^3_g\|_{\dyprodhp\to \dyprodhq} \simeq \|g\|_{L^r(\mathbb{R}^2)}.$$	
\end{theorem}	
\begin{proof}
First we treat \eqref{the firstcassss}, and we begin by proving the lower bound for the operator norm. So normalize $g$, and assume that $\|\pi^3_g\|_{\dyprodhp\to \dyprodhp}=1$. Then fix $J$ and take an arbitrary function of the form
\[
f(x,y)=f_J(x)h_J(y),
\]
to which applying the operator yields
\[
\pi^3_g(f)=\sum_{I\in \mathcal{D}}\avr{f_J}{I} \avr{g_I}{J} h_I(x)h_J(y).
\]
Therefore, boundedness of $\pi^3_g$, implies that the family of one-parameter paraproducts
\begin{equation}\label{averaginJ}
	\pi_{g'}(b)=\sum_{I\in \mathcal{D}}\avr{b}{I} \avr{g_I}{J} h_I, \quad g'=\sum_{I\in \mathcal{D}} \avr{g_I}{J} h_I,
\end{equation}
are uniformly bounded in $J$. So we have
\[
\sum_{I'\subseteq I}\avr{g_{I'}}{J}^2\lesssim |I|, \quad I\in\mathcal{D},
\]
which after using Lebesgue's differentiation theorem implies
\[
\sum_{I'\subseteq I}g_{I'}(y)^2\lesssim |I|, \quad I\in\mathcal{D} \quad \text{a.e}\quad y,
\]
as claimed in \eqref{the firstcassss}.\\

Next, we prove the upper bound for $\pi^3_g$, and to this end again normalize $g$, such that
\[
\sum_{I'\subseteq I}g_{I'}(y)^2\le |I|, \quad I\in\mathcal{D} \quad \text{a.e}\quad y,
\]
implying the uniform $L^2(\mathbb{R})$-boundedness of operators in \eqref{averaginJ}, which means that for an arbitrary function of two variables $f$, and each $J$ we have
\[
\sum_{I\in \mathcal{D}}\avr{f_J}{I}^2 \avr{g_I}{J}^2\lesssim \|f_J\|_{L^2(\mathbb{R})}^2.
\]
Summing over $J$ gives 
\[
\|\pi^3_g(f)\|_{L^2(\mathbb{R}^2)}^2=\sum_{I,J\in\mathcal{D}}\avr{f_J}{I}^2 \avr{g_I}{J}^2\lesssim \sum_{J\in \mathcal{D}} \|f_J\|_{L^2(\mathbb{R})}^2=\|f\|_{L^2(\mathbb{R}^2)}^2,
\]
establishing $L^2(\mathbb{R}^2)$-boundedness of $\pi^3_g$.
Next we note that $\pi^3_g$ is a local operator simply because for any dyadic rectangle $R$, and any two dyadic intervals $I,J$, $\avr{h_R,\bar{\chi}_I\otimes h_J}{}\neq 0$, only if $I\times J\subset R$. Therefore, according to Lemma \ref{locallemma} it is enough to show that for any $1<q<s<\infty$, and any $L^s$-atom $f$ with support on $\Omega$, $\pi^3_g(f)$ is an $L^q$-atom. For simplicity assume that $|\Omega|=1$, and for each $I$, and $y$ let 
\[
\Omega_I=\{y: I\times \{y\}\subseteq \Omega\}, \quad \Omega_y=\Omega\cap \mathbb{R}\times \{y\}.
\]
Next, let $\frac{1}{q}=\frac{1}{s}+\frac{1}{t}$ and note that
\begin{align*}
	&S(\pi^3_g(f))(x,y)=\big(\sum_{I\times J\subseteq \Omega}\avr{f_J}{I}^2\avr{g_I}{J}^2\bar{\chi}_I(x)\bar{\chi}_J(y)\big)^{\frac{1}{2}}\le\\
	&\big(\sum_{J\in\mathcal{D}}M(f_J)^2(x)\bar{\chi}_J(y)\big)^{\frac{1}{2}} \big(\sum_{I\in\mathcal{D}}M(g_I\chi_{\Omega_I})^2(y)\bar{\chi}_I(x)\big)^{\frac{1}{2}}.
\end{align*}
Then H\"older's inequality combined with $L^s(\mathbb{R}^2)$-boundedness of $S_2M_1$ gives us
\[
\|S(\pi^3_g(f))\|_{L^q(\mathbb{R}^2)}^t\lesssim\iint\big(\sum_{I\in\mathcal{D}}M(g_I(y)\chi_{\Omega_I})^2(y)\bar{\chi}_I(x)\big)^{\frac{t}{2}}dy dx.
\]
Now fix $x$, and since $t>1$, we may apply Fefferman-Stein inequality and get
\[
\|S(\pi^3_g(f))\|_{L^q(\mathbb{R}^2)}^t\lesssim\iint\big(\sum_{I\in\mathcal{D}}g_I(y)^2\chi_{\Omega_I}(y)\bar{\chi}_I(x)\big)^{\frac{t}{2}}dydx,
\]
which combined with Fubini again implies
\[
\|S(\pi^3_g(f))\|_{L^q(\mathbb{R}^2)}^t\lesssim \int dy\int \big(\sum_{I\subseteq \Omega_y}g_I(y)^2\bar{\chi}_I(x)\big)^{\frac{t}{2}}dx.
\]
Finally, we use our assumption on $g$, and estimate the inner integral by the John-Nirenberg lemma and obtain
\[
\|S(\pi^3_g(f))\|_{L^q(\mathbb{R}^2)}^t\lesssim\int|\Omega_y|dy=|\Omega|=1,
\]
which completes the proof of \eqref{the firstcassss}.\\

Now we turn to \eqref{thesecondcasssss}. Again normalize $g$ such that $\|\pi^3_g\|_{\dyprodhp\to \dyprodhq}=1$, then our task is to show that
\[
\left\|\|g(x,y)\|_{H_d^r(\mathbb{R}, dx)}\right\|_{L^r(\mathbb{R}, dy)}^r\simeq\iint\big(\sum_{I\in\mathcal{D}}g_I(y)^2\bar{\chi}_I(x)\big)^\frac{r}{2}dxdy\lesssim 1.
\]
 To this end, for a fixed $J$ consider the operator $\pi_{g'}$ defined in \eqref{averaginJ} and note that since
\[
\|\pi_{g'}\|_{H_d^p(\mathbb{R})\to H_d^q(\mathbb{R})} \simeq \|g'\|_{H_d^r(\mathbb{R})},
\]
we may find a single variable function $f_J$ such that
\[
\|f_J\|_{H_d^p(\mathbb{R})}^p=\|g'\|_{H_d^r(\mathbb{R})}^r, \quad \|g'\|_{H_d^r(\mathbb{R})}\lesssim \frac{\|\pi_{g'}(f_J)\|_{ H_d^q(\mathbb{R})}}{\|f_J\|_{H_d^p(\mathbb{R})}},
\]
implying that
\[
\|g'\|_{H_d^r(\mathbb{R})}^r\lesssim \|\pi_{g'}(f_J)\|_{ H_d^q(\mathbb{R})}^q.
\]
Therefore, for each $J$ we must have
\begin{equation}
	\int M(f_J)^p(x)dx\lesssim\int\big(\sum_{I\in\mathcal{D}}\avr{g_I}{J}^2\bar{\chi}_I(x)\big)^\frac{r}{2}dx\lesssim \int\big(\sum_{I\in\mathcal{D}}\avr{f_J}{I}^2\avr{g_I}{J}^2\bar{\chi}_I(x)\big)^\frac{q}{2}dx.
\end{equation}
Now, multiply sides of the above inequality to $\chi_J(y)$, sum over an arbitrary finite disjoint collection of intervals $\mathcal{C}$, and integrate in $y$ to get
\begin{align}
	&\iint \sum_{J\in\mathcal{C}} M(f_J)^p(x)\chi_J(y)dxdy\lesssim \int\sum_{J\in\mathcal{C}}\big(\sum_{I\in\mathcal{D}}\avr{g_I}{J}^2\bar{\chi}_I(x)\big)^\frac{r}{2}\chi_J(y)dxdy\label{tripleinequlaity1111}\\
	&\lesssim \iint\sum_{J\in\mathcal{C}} \big(\sum_{I\in\mathcal{D}}\avr{f_J}{I}^2\avr{g_I}{J}^2\bar{\chi}_I(x)\big)^\frac{q}{2}\chi_J(y)dxdy.\label{tripleinequlaity}
\end{align}
Then note that since $\mathcal{C}$ is disjoint we may write the last integrand as
\[
\sum_{J\in\mathcal{C}} \big(\sum_{I\in\mathcal{D}}\avr{f_J}{I}^2\avr{g_I}{J}^2\bar{\chi}_I(x)\big)^\frac{q}{2}\chi_J(y)= \big(\sum_{J\in\mathcal{C}}\sum_{I\in\mathcal{D}}\avr{f_J}{I}^2\avr{g_I}{J}^2\bar{\chi}_I(x)\chi_J(y)\big)^\frac{q}{2},
\]
or equivalently
\[
\sum_{J\in\mathcal{C}} \big(\sum_{I\in\mathcal{D}}\avr{f_J}{I}^2\avr{g_I}{J}^2\bar{\chi}_I(x)\big)^\frac{q}{2}\chi_J(y)=S(\pi^3_g(f))^q(x,y), \quad f(x,y)=\sum_{J\in\mathcal{C}}|J|^{\frac{1}{2}}f_J(x)h_J(y).
\]
And for the same reason, may write the first integrand in \eqref{tripleinequlaity1111} as
\[
\sum_{J\in\mathcal{C}} M(f_J)^p(x)\chi_J(y)=\big(\sum_{J\in\mathcal{C}} M(f_J)^2(x)\chi_J(y)\big)^\frac{p}{2}=S_2M_1(f)^p(x,y).
\]
Plugging these into \eqref{tripleinequlaity1111}, \eqref{tripleinequlaity} and applying boundedness of $\pi^3_g$ gives us
\[
\dyprodhpn{f}^p\lesssim \iint\sum_{J\in\mathcal{C}}\big(\sum_{I\in\mathcal{D}}\avr{g_I}{J}^2\bar{\chi}_I(x)\big)^\frac{r}{2}\chi_J(y)dxdy \lesssim \dyprodhpn{f}^q,
\]
which implies that
\[
\iint\sum_{J\in\mathcal{C}}\big(\sum_{I\in\mathcal{D}}\avr{g_I}{J}^2\bar{\chi}_I(x)\big)^\frac{r}{2}\chi_J(y)dxdy\lesssim 1.
\]
Now, using Fubini first, and since $\mathcal{C}$ is arbitrary, applying  Lebesgue's differentiation theorem together with Fatou's lemma yields
\[
\iint\big(\sum_{I\in\mathcal{D}}g_I(y)^2\bar{\chi}_I(x)\big)^\frac{r}{2}dydx\lesssim 1,
\]
which completes the proof of \eqref{thesecondcasssss}, and the theorem.	
\end{proof}	
\begin{remark}
	A similar argument to the one used above shows that the mixed norm in \eqref{the firstcassss} is stronger than the product BMO norm.
\end{remark}

\subsection{The Operator $\pi^4_g$}	
Finally, we turn to the last operator, $\pi^4_g$, defined by
\[
\pi^4_g(f)=\sum_{I,J\in\mathcal{D}}\avr{f_J}{I}g_{I\times J} h_I\otimes \bar{\chi}_J,
\]	
which makes it very different from the previous ones. The reason is that, unlike the other operators, $\pi^4_g$ and its adjoint are given as an expansion in an overdetermined system rather than a basis for $L^2(\mathbb{R}^2)$. Therefore, one cannot directly estimate the square function of $\pi^4_g(f)$, and has to appeal to duality. Meaning that we must take $f'$, pair it with $\pi^4_g(f)$, and note that
\[
\avr{\pi^4_g(f), f'}{}=\sum_{I,J\in\mathcal{D}}\avr{f_J}{I}\avr{f'_I}{J}g_{I\times J} =\avr{\pi^3_{f'}(f),g}{},
\]
and therefore, from the boundedness of $\pi^3_{f'}$ and duality it follows that
\begin{align}
	&\|\pi^4_g\|_{L^p(\mathbb{R}^2)\to L^p(\mathbb{R}^2)}\lesssim \dyprodbmon{g}, \quad 1< p<\infty,\\
	&\|\pi^4_g\|_{L^p(\mathbb{R}^2)\to L^q(\mathbb{R}^2)}\lesssim \|g\|_{L^r(\mathbb{R}^2)}, \quad \frac{1}{q}=\frac{1}{p}+\frac{1}{r}, \quad 1<p,r,q<\infty.
\end{align}

However, there is another approach, which we explain in the diagonal case. So, fix $1<p<\infty$, and for a function $g$ with finite Haar support, let
\begin{equation}\label{normpppp}
	\|g\|_p:=\|\pi^4_g\|_{L^p(\mathbb{R}^2)\to L^p(\mathbb{R}^2)},
\end{equation}
which is identical to
\[
\|g\|_p=\sup_{\substack{\|f\|_{L^p(\mathbb{R}^2)}=1\\\|f'\|_{L^{p'}(\mathbb{R}^2)}=1}}\Big|\sum_{I,J\in\mathcal{D}}\avr{f_J}{I}\avr{f'_I}{J}g_{I\times J}\Big|.
\]
Then, write the sum as 
\begin{equation}\label{pi4identity}
	\sum_{I,J\in\mathcal{D}}\avr{f_J}{I}\avr{f'_I}{J}g_{I\times J}=\iint \sum_{I,J\in\mathcal{D}}g_{I\times J}|I\times J|^{-\frac{1}{2}}f_J(x)|J|^{-\frac{1}{2}}f'_I(y)|I|^{-\frac{1}{2}}\chi_I(x)\chi_J(y)\, dxdy,
\end{equation}
and let $G(x,y)$ be the ``$\mathcal{D}$ by $\mathcal{D}$'' matrix defined by
\begin{equation}\label{GGGGGG}
	G(x,y)_{I,J}:=g_{I\times J}|I\times J|^{-\frac{1}{2}}\chi_I(x)\chi_J(y), \quad I,J\in\mathcal{D},
\end{equation}
and similarly define the two vectors in $l^2(\mathcal{D})$ by
\[
\vec{f}(x,y)_J:=f_J(x)|J|^{-\frac{1}{2}}\chi_J(y), \quad J\in\mathcal{D}, \quad 
\vec{f'}(x,y)_I:=f'_I(y)|I|^{-\frac{1}{2}}\chi_I(x), \quad I\in\mathcal{D}.
\]
Then note that the integrand in \eqref{pi4identity} can be written as
\[
\sum_{I,J\in\mathcal{D}}G(x,y)_{I,J}\vec{f}(x,y)_J\vec{f'}(x,y)_I = \avr{G(x,y)\vec{f}(x,y),\vec{f'}(x,y)}{},
\]
where $G(x,y)\vec{f}(x,y)$ is understood as the multiplication of a matrix with a vector, and the inner product is in $l^2(\mathcal{D})$. From this point of view, we may bound the right-hand side of \eqref{pi4identity} by
\[
\Big|\sum_{I,J\in\mathcal{D}}\avr{f_J}{I}\avr{f'_I}{J}g_{I\times J}\Big| \le \iint \big|\avr{G(x,y)\vec{f}(x,y),\vec{f'}(x,y)}{}\big|\, dxdy,
\]
which implies that
\begin{align*}
	\Big|\sum_{I,J\in\mathcal{D}}\avr{f_J}{I}\avr{f'_I}{J}g_{I\times J}\Big|
	&\le \iint \|G(x,y)\|_{l^2(\mathcal{D})\to l^2(\mathcal{D})} \|\vec{f}(x,y)\|_{l^2(\mathcal{D})} \|\vec{f'}(x,y)\|_{l^2(\mathcal{D})} \, dxdy \\
	&\le \Big\|\|G(\cdot)\|_{l^2(\mathcal{D})\to l^2(\mathcal{D})}\Big\|_{L^{\infty}(\mathbb{R}^2)} 
	\iint \|\vec{f}(x,y)\|_{l^2(\mathcal{D})} \|\vec{f'}(x,y)\|_{l^2(\mathcal{D})}\, dxdy \\
	&\le \Big\|\|G(\cdot)\|_{l^2(\mathcal{D})\to l^2(\mathcal{D})}\Big\|_{L^{\infty}(\mathbb{R}^2)} 
	\Big\|\|\vec{f}(\cdot)\|_{l^2(\mathcal{D})}\Big\|_{L^p(\mathbb{R}^2)} 
	\Big\|\|\vec{f'}(\cdot)\|_{l^2(\mathcal{D})}\Big\|_{L^{p'}(\mathbb{R}^2)}.
\end{align*}
Then note that, since
\[
\Big\|\|\vec{f}(\cdot)\|_{l^2(\mathcal{D})}\Big\|_{L^p(\mathbb{R}^2)} \simeq \|f\|_{L^p(\mathbb{R}^2)}, \quad 
\Big\|\|\vec{f'}(\cdot)\|_{l^2(\mathcal{D})}\Big\|_{L^{p'}(\mathbb{R}^2)} \simeq \|f'\|_{L^{p'}(\mathbb{R}^2)},
\]	
we obtain
\begin{equation}\label{thenewbound}
	\|g\|_p \lesssim \Big\|\|G(\cdot)\|_{l^2(\mathcal{D})\to l^2(\mathcal{D})}\Big\|_{L^{\infty}(\mathbb{R}^2)},
\end{equation}	
a new bound for $\|g\|_p$.\\

Now, if $g$ has a tensorial form $g=b\otimes c$, then we simply have
\[
\|G(x,y)\|_{l^2(\mathcal{D})\to l^2(\mathcal{D})}=\big(\sum_{I\in\mathcal{D}}b_I^2\bar{\chi}_I(x)\big)^{\frac{1}{2}}\big(\sum_{J\in\mathcal{D}}c_J^2\bar{\chi}_J(y)\big)^{\frac{1}{2}}=S(b)(x)S(c)(y)=S(g)(x,y).
\]
Also, because of the tensorial structure of the operator we get that
\[
\|g\|_p\simeq \|b\|_{BMO_d(\mathbb{R})}\|c\|_{BMO_d(\mathbb{R})}=\dyprodbmon{g},
\]	
which can be much smaller than 	$\|S(g)\|_{L^{\infty}(\mathbb{R}^2)}$, and  thus the bound in \eqref{thenewbound} is not sharp. Nevertheless, as the next example shows, sometimes this bound is much better than the $\dyprodbmo$-norm.\\

\textbf{Example.}
 For a dyadic interval $I$ in $[0,1]$ with $|I|=2^{-m}$ let, $i(I)=m$. Now fix a large number $n$, let 
\[
g'_{I\times J}=\begin{cases}
	|I\times J|^{\frac{1}{2}}& 0\le i(I), i(J)<n,\\
	0 & \text{otherwise},
\end{cases}
\]
and note that the dyadic rectangles coming from the product of the first $n$ generations of dyadic intervals in $[0,1]$, cover the unit square $n^2$ times. Therefore, we have $\dyprodbmon{g'}\ge n$. Indeed since $g'$ has a tensorial structure we have $\dyprodbmon{g'}= n$. Also, note that changing the signs of coefficients of $g'$ does not change the $\dyprodbmo$-norm of the function. Next, take an $n\times n$ matrix $\{H_{ij}: 0\le i,j<n\}$, with $\pm1$ entries and with small norm. A perfect example is a Hadamard matrix, which is simply an orthogonal matrix with $\pm1$ entries and therefore its operator norm on $\rn$ equipped with the Euclidean norm is exactly $\sqrt{n}$. Such a matrix exists when $n$ is a power of 2. Now we modify $g'$, obtain a new function $g$ defined by
\[
g_{I\times J}=H_{i(I),i(J)}g'_{I\times J}, \quad I,J\in \mathcal{D},
\]
and note that for each $(x,y)$, the matrix $G(x,y)$ is a copy of the matrix $H$. Therefore,
\[
\|G(x,y)\|_{l^2(\mathcal{D})\to l^2(\mathcal{D})}=\sqrt{n},
\]	
 which combined with \eqref{thenewbound} implies
\[
\|g\|_p\lesssim \sqrt{n},
\]	
showing that $\|g\|_p$ is much smaller than $\dyprodbmon{g}$.\\

Interestingly, this example shows that the rectangular Haar basis is not an unconditional basis for the space of functions equipped with the $\|.\|_p$-norm. So, one may wonder whether the remedy is to consider a stronger norm defined as
\[
\|g\|_p':=\sup_{\substack{\|f\|_{L^p(\mathbb{R}^2)}=1\\\|f'\|_{L^{p'}(\mathbb{R}^2)}=1}}\sum_{I,J\in\mathcal{D}}\avr{|f_J|}{I}\avr{|f'_I|}{J}|g_{I\times J}|.
\]
However, a modification of the above example shows that, again, this quantity can be much smaller than the $\dyprodbmo$-norm. To see this, simply let $H$ be the identity matrix, implying that $g_{I\times J}=|I\times J|^{\frac{1}{2}}$ on the first $n$ generations of dyadic squares in $[0,1]^2$. Since they cover the unit square $n$ times, we must have $\dyprodbmon{g}\ge \sqrt{n}$. But now, at each point, the operator norm of the matrix $G(x,y)$ is $1$, and thus $\|g\|_p'\lesssim 1$, again a quantity much smaller than the $\dyprodbmo$-norm. Unfortunately, we were unable to combine this new argument with the previous ones invoking the atomic decomposition, and could not obtain an improved product BMO-type bound for $\pi^4_g$. So let us pose the following question:
\begin{question}
	Is there a reasonable characterization of the operator norm of $\pi^4_g$? In particular, is there a way to combine the above matrix point of view with the atomic decomposition techniques to provide an improved bound for this operator?
\end{question}

It is notable that sometimes the norm defined in \eqref{normpppp} is independent of $p$. First, note that although the operator $\pi^4_g$ destroys cancellation of atoms, it does not enlarge their support. Therefore, Lemma \ref{locallemma} implies that
\[
\|g\|_p\lesssim \|g\|_q, \quad 1<p\le q<\infty.
\]
Then, under the condition that 
\[
g_{I\times J}=g_{J\times I}, \quad I,J\in\mathcal{D},
\]
there holds
\[
\|g\|_p\simeq \|g\|_q, \quad 1<p,q<\infty.
\]
To see this, let $\tilde{f}(x,y)=f(y,x)$, and then we note that
\[
(\pi^4_g)^{t}(\tilde{f})(y,x)=\pi^4_g(f)(x,y),
\]	
which implies $\|g\|_r=\|g\|_{r'}$, for any $1<r<\infty$. Therefore, if we choose $1<r<\min\{p,p',q,q'\}$, we must have $\|g\|_p\simeq \|g\|_q\simeq \|g\|_r$, proving our claim.\\

We conclude this note by mentioning a standard application of the lower bounds in the weak factorization theorems for Hardy spaces. Here, a basic tool is a general atomic decomposition stating that:

\begin{theorem*}\label{atomiclemma}
	Let $V \subset X$ be a bounded subset of a Banach space $X$, symmetric around the origin, and such that for every $l \in X^*$,
	\[
	\|l\|_{X^*} \simeq \sup_{x \in V} |l(x)|.
	\]
	Then, for every $x \in X$, there exists a sequence of positive real numbers $\{\lambda_i\}_{i \geq 1}$ such that
	\[
	x = \sum_{i \geq 1} \lambda_i v_i, \quad \|x\|_X \simeq \sum_{i \geq 1} \lambda_i, \quad v_i \in V, \quad i \geq 1.
	\]
	In the above, the series converges in the norm topology of $X$.
\end{theorem*}
See \cite{MR1225511,MR4338459} for the proof and applications of this result. Now, as an example, the bound
\[
\|\pi^3_g\|_{L^p(\mathbb{R}^2)\to L^q(\mathbb{R}^2)} \simeq \|g\|_{L^r(\mathbb{R}^2)}, \quad \frac{1}{q}=\frac{1}{p}+\frac{1}{r}, \quad 1<p,r,q<\infty,
\]	
can be viewed as
\[
\|g\|_{L^r(\mathbb{R}^2)}\simeq \sup_{\substack{\|f\|_{L^p(\mathbb{R}^2)}=1\\\|f'\|_{L^{q'}(\mathbb{R}^2)}=1}}|\avr{\pi^4_{f'}(f), g}{}|,
\]	
and thus it follows directly from the above theorem that, for any function $g'\in L^{r'}(\mathbb{R}^2)$, there exists a sequence of positive numbers $\{\lambda_i\}_{i\ge 0}$, and pairs of functions $\{f_i, f'_i\}_{i\ge0}$ in the unit balls of $L^p(\mathbb{R}^2)$ and $L^{q'}(\mathbb{R}^2)$, such that
\[
g'=\sum_{i}\lambda_i\pi^4_{f'_i}(f_i), \quad \|g'\|_{L^{r'}(\mathbb{R}^2)}\simeq \sum_{i}\lambda_i.
\]
On the other hand, the converse statement to the above theorem holds as well, and as our earlier example shows
\[
\dyprodbmon{g}\not\simeq\sup_{\substack{\|f\|_{L^p(\mathbb{R}^2)}=1\\\|f'\|_{L^{p'}(\mathbb{R}^2)}=1}}|\avr{\pi^3_{f'}(f), g}{}|.
\]	
Thus, there exists a function $g'\in H_d^1(\mathbb{R}\otimes\mathbb{R})$ for which any factorization of the form
\[
g'=\sum_{i}\lambda_i \pi^3_{f'_i}(f_i), \quad \|g'\|_{H_d^1(\mathbb{R}\otimes\mathbb{R})}\simeq \sum_{i}\lambda_i, \quad \|f_i\|_{L^{p}(\mathbb{R}^2)}=1, \quad \|f'_i\|_{L^{p'}(\mathbb{R}^2)}=1,
\]	
would fail to hold. In particular, $g'$ is not of the form $\pi^3_{f'}(f)$, for any pair of functions in the unit ball of $L^2(\mathbb{R}^2)$. Equivalently, for any two families of functions on line $\{f_J\}_{J\in\mathcal{D}}$ and $\{f'_I\}_{I\in\mathcal{D}}$ satisfying 
\[
\sum_{J \in \mathcal{D}}\|f_J\|_{L^2(\mathbb{R}^2)}^2=1=\sum_{I \in \mathcal{D}}\|f'_I\|_{L^2(\mathbb{R}^2)}^2,
\]
 the Haar coefficients of $g'$ cannot be of the form
\[
g'_{I\times J}=\avr{f_J\otimes f'_I}{I\times J}, \quad I,J\in \mathcal{D}.
\]

	\section*{Acknowledgments}
	I would like to thank the referee for valuable comments on the manuscript.


\end{document}